\documentclass[]{amsart}

\usepackage{amsmath,amsthm,amssymb}
\usepackage{wasysym}
\usepackage{txfonts}
\usepackage{mathrsfs}
\usepackage[all,2cell]{xy}
\UseTwocells
\usepackage{tikz}
\usepackage{diagbox}
\usepackage{hyperref}

\newtheorem{theorem}{Theorem}[section]
\newtheorem{proposition}[theorem]{Proposition}
\newtheorem{lemma}[theorem]{Lemma}
\newtheorem{corollary}[theorem]{Corollary}

\newtheorem{conjecture}[theorem]{Conjecture}

\theoremstyle{definition}
\newtheorem{definition}[theorem]{Definition}

\newtheorem{example}[theorem]{Example}
\newtheorem{remark}[theorem]{Remark}

\title{A combinatorial approach to Kontsevich's Swiss cheese conjecture}

\author[Florian De Leger]{Florian De Leger}

\email{fdeleger@mail.ntua.gr}

\begin{document}
	
\begin{abstract}
	From a coloured operad $\mathcal{P}$ and a $\mathcal{P}$-algebra $A$, we construct a new operad $\mathrm{SC}(\mathcal{P})$ and a Hochschild object $\mathrm{Hoch}(A)$ together with an $\mathrm{SC}(\mathcal{P})$-action on the pair $(\mathrm{Hoch}(A),A)$. We prove that if $\mathcal{P}$ is the little $n$-disks operad, then $\mathrm{SC}(\mathcal{P})$ is equivalent to the Swiss cheese operad $\mathcal{SC}_n$.  This gives us a weak version of Kontsevich's Swiss cheese conjecture (without the universal property). We apply our result to prove the existence of an $E_{n+1}$-action on the Hochschild-Pirashvili cochain of order $n$ of a commutative algebra.
\end{abstract}
	
\maketitle

\tableofcontents

\section*{Introduction}

The \emph{Swiss cheese conjecture} was proposed by Kontsevich \cite[Claim 1]{kontsevich}. It states that for any $n$-algebra $A$ (that is, an algebra over the little $n$-disks operad), there is a \emph{universal} $n+1$-algebra $\mathrm{Hoch}(A)$ which \emph{acts} on $A$. Here, we say that a $n+1$-algebra $B$ \emph{acts} on a $n$-algebra $A$ if there is an action of the Swiss cheese operad $\mathcal{SC}_n$ \cite{voronov} on the pair $(B,A)$, extending the existing algebra structures on $A$ and $B$. $\mathrm{Hoch}(A)$ is \emph{universal} in the sense that it is the terminal object in the homotopy category of $n+1$-algebras equipped with an action on $A$. If $n=1$ and $A$ is an associative algebra, then $\mathrm{Hoch}(A)$ is the classical Hochschild cochain of $A$ and one recovers Deligne's conjecture.

Despite its name, several proofs of Kontsevich's conjecture already exist \cite{bataninshoikhet,dolgushev,hkv,thomas}. In \cite{dolgushev}, the authors defined the operad $\O$ whose algebras are multiplicative non-symmetric operads $\mathcal{X}$ equipped with an algebra over $\mathcal{X}$. They then proved \cite[Theorem 2.1]{dolgushev} that the condensation of $\O$, as defined in \cite[Section 1]{BataninBergerLattice}, yields Voronov's Swiss cheese operad. This implies Kontsevich's conjecture for $n=1$.

Our goal is to extend their approach to any $n \geq 1$. To this end, we will extend the construction of the operad $\O$, using the Baez-Dolan \emph{plus construction} \cite{baezdolan}. More precisely, for a coloured categorical operad $\mathcal{P}$, we will define the categorical operad $\mathcal{P}_{sc}^+$ whose algebras are $\mathcal{P}^+$-algebras $\mathcal{X}$, where $\mathcal{P}^+$ is the Baez-Dolan plus construction, equipped with an \emph{internal $\mathcal{P}^+$-algebra} \cite[Definition 7.2]{batanin} and an algebra over $\mathcal{X}$. When $\mathcal{P}=\mathcal{A}ss$ is the operad for associative algebras, $\mathcal{P}^+_{sc}$ is the lax version of the operad $\O$.

The main result of this paper is Theorem \ref{theoremswisscheese}. It states that, when $\mathcal{P}$ is the little $n$-disks operad, the condensation of $\mathcal{P}_{sc}^+$, written $\mathrm{SC}(\mathcal{P})$, is equivalent to the $n+1$-dimensional Swiss cheese operad. Our method to get this result is extremely similar to the one used in \cite{BataninBergerLattice}. We will work with the \emph{relative complete graph operad} defined in \cite[Section 3.2]{quesney}. It is the categorical version of Voronov's topological Swiss cheese operad, in the sense that applying the classifying space functor pointwise to it yields the Swiss cheese operad. The main ingredient in our proof of Theorem \ref{theoremswisscheese} is the construction of a lax morphism of categorical operads involving the relative complete graph operad. This lax morphism is called \emph{complexity map} by analogy with the morphism defined in \cite[Proposition 3.4]{BataninBergerLattice}. We then prove that this complexity map induces the desired weak equivalence between operads. The intuitive idea behind this weak equivalence is that the Baez-Dolan plus construction, which involves trees, creates an extra dimension. Applying Theorem \ref{theoremswisscheese}, we will prove the existence of an $E_{n+1}$-action on the Hochschild-Pirashvili cochain of order $n$ of a commutative algebra. This answers a question raised in \cite[Remark 2.19]{BataninBergerLattice}.

This paper is organised as follows.

In Section \ref{sectionpreliminaries}, we recall the notions of coloured operads and their algebras. We also recall the notion of a \emph{Swiss cheese type} operad, which is a particular case of coloured operad. We give two important examples of Swiss cheese type operads: Voronov's Swiss cheese operad and Quesney's relative complete graph operad. Finally, we recall two constructions for coloured operads: the Baez-Dolan plus construction and Batanin-Berger's condensation. Both constructions are needed to define our Swiss cheese construction. Note that we slightly extend each construction. Indeed, we consider a categorical version of the Baez-Dolan plus construction. Regarding the condensation of a coloured operad, we include the case where the condensed operad may be multi-coloured. More precisely, for a $C$-coloured operad $\mathcal{P}$ and a function $f: C \to D$, we define a $D$-coloured operad as the \emph{condensation of $\mathcal{P}$ along $f$}. We recover the classical definition when $D$ is the singleton set and $f$ is the unique map.

In Section \ref{sectionswisscheeseconstruction}, we introduce a Swiss cheese construction $\mathrm{SC}(\mathcal{P})$ for a coloured operad $\mathcal{P}$. We also define a Hochschild object $\mathrm{Hoch}(A)$ for a $\mathcal{P}$-algebra $A$. We prove that the $\mathcal{P}$-action on $A$ can be extended to an $\mathrm{SC}(\mathcal{P})$-action on the pair $(\mathrm{Hoch}(A),A)$. We conjecture a universal property for $\mathrm{Hoch}(A)$.

Section \ref{sectionkontsevichconjecture} is dedicated to proving Theorem \ref{theoremswisscheese}.

Finally, Section \ref{sectionhochschildpirashvili} concerns the Hochschild-Pirashvili cochains. We prove that our definition of $\mathrm{Hoch}(A)$ coincides with this classical notion when $A$ is a commutative algebra. Thanks to Theorem \ref{theoremswisscheese}, this implies the existence of an $E_{n+1}$-action on the Hochschild-Pirashvili cochains of order $n$.

\subsection*{Acknowledgements}

I want to thank Vasileios Aravantinos-Sotiropoulos, Hyeon Tai Jung and Christina Vasilakopoulou for our interesting discussions on this topic. I acknowledge that this work was implemented in the framework of H.F.R.I call “3rd Call for H.F.R.I.’s Research Projects to
Support Faculty Members \& Researchers” (H.F.R.I. Project Number: 23249).

\section{Preliminaries}\label{sectionpreliminaries}

\subsection{Coloured operads and their algebras}\label{subsectioncolouredoperads}

\begin{definition}
	A \emph{coloured operad} $\mathcal{P}$ in a symmetric monoidal category $(\mathcal{E},\otimes,e)$ is given by
	\begin{itemize}
		\item a set of \emph{colours} $C$,
		\item for $k \geq 0$ and $(c_1,\ldots,c_k,c) \in C^{k+1}$, an object
		\[
			\mathcal{P}(c_1,\ldots,c_k;c) \in \mathcal{E},
		\]
		\item for $c \in C$, a map
		\[
			e \to \mathcal{P}(c;c)
		\]
		called \emph{unit},
		\item for $1 \leq i \leq k$, $j \geq 0$, $(c_1,\ldots,c_k,c) \in C^{k+1}$ and $(c'_1,\ldots,c'_j) \in C^j$, a map
		\begin{equation}\label{compositionmap}
			\circ_i: \mathcal{P}(c_1,\ldots,c_k;c) \otimes \mathcal{P}(c'_1,\ldots,c'_j;c_i) \to \mathcal{P}(c_1,\ldots,c_{i-1},c'_1,\ldots,c'_j,c_{i+1},\ldots,c_k;c),
		\end{equation}
		called \emph{multiplication},
		\item for $k \geq 0$ and $(c_1,\ldots,c_k,c) \in C^{k+1}$, an action of the symmetric group
		\[
		\sigma \in \Sigma_k \mapsto \left[\mathcal{P}(c_1,\ldots,c_k;c) \to \mathcal{P}(c_{\sigma^{-1}(1)},\ldots,c_{\sigma^{-1}(k)};c)\right],
		\]
	\end{itemize}
	satisfying associativity, unitality and equivariance axioms.
\end{definition}

\begin{definition}
	A \emph{topological} (resp. \emph{categorical}) operad is a coloured operad in the category of topological spaces (resp. small categories).
\end{definition}

\begin{example}[The endomorphism operad]\label{exampleendomorphismoperad}
	Let $C$ be a set, $\mathcal{M}$ a symmetric monoidal $\mathcal{E}$-category and $A := (A_c)_{c \in C}$ a $C$-indexed collection of objects in $\mathcal{M}$. Let $\mathrm{End}_A$ be the $C$-coloured operad in $\mathcal{E}$ given, for $(c_1,\ldots,c_k;c) \in C^{k+1}$, by
	\[
		\mathrm{End}_A(c_1,\ldots,c_k;c) := \mathcal{M}(A_{c_1} \otimes \ldots \otimes A_{c_k},A_c).
	\]
	The unit is given by the identity and the multiplication is given by composition.
\end{example}

Note that there is an obvious notion of a \emph{morphism of coloured operads}, such that coloured operads form a category.

\begin{definition}
	An algebra $A$ in $\mathcal{M}$ over a $C$-coloured operad $\mathcal{P}$ in $\mathcal{E}$ is given by a $C$-indexed collection $A := (A_c)_{c \in C}$ of objects in $\mathcal{M}$, together with a morphism of operad
	\[
		\mathcal{P} \to \mathrm{End}_A.
	\]
\end{definition}

\subsection{Swiss cheese type operads}

\begin{definition}\cite[Definition 2.11]{quesney}
	A \emph{Swiss cheese type operad} is a coloured operad $\mathcal{P}$ with set of colours $C = C_\mathsf{f} \sqcup C_\mathsf{h}$ such that $\mathcal{P}(c_1,\ldots,c_k;c)$ is empty if $c \in C_\mathsf{f}$ and $c_i \in C_\mathsf{h}$ for some $i \in \{1,\ldots,k\}$.
\end{definition}

\begin{definition}\cite[Definition 8]{kontsevich}\label{definitionswisscheeseaction}
	Let $\mathcal{P}$ be a Swiss cheese type operad with set of colours $C = C_\mathsf{f} \sqcup C_\mathsf{h}$. Let $\mathcal{P}_\mathsf{f}$ and $\mathcal{P}_\mathsf{h}$ be the full suboperads of $\mathcal{P}$ restricted on the sets of colours $C_\mathsf{f}$ and $C_\mathsf{h}$ respectively. A \emph{Swiss cheese action} of an $\mathcal{P}_\mathsf{f}$-algebra $B$ on an $\mathcal{P}_\mathsf{h}$-algebra $A$ is an $\mathcal{P}$-algebra structure on the pair $(B,A)$ which extends the structures of algebras on $B$ and $A$.
\end{definition}

\subsection{Voronov's Swiss cheese operad}\label{sectionvoronov}

\begin{definition}\cite[Section 1]{voronov}
	For $n \geq 1$, the \emph{Swiss cheese operad} $\mathcal{SC}_n$ is a Swiss cheese type topological operad. Its set of colours is $C := \{\mathsf{f},\mathsf{h}\}$. For $k \geq 0$,
	\[
		\mathcal{SC}_n(\underbrace{\mathsf{f},\ldots,\mathsf{f}}_k;\mathsf{f}) = \mathcal{D}_{n+1}(k)
	\]
	is the space of ordered configurations of $k$ disjoint little disks inside the unit $n+1$-dimensional disk. For $(c_1,\ldots,c_k) \in C^k$,
	\[
		\mathcal{SC}_n(c_1,\ldots,c_k;\mathsf{h})
	\]
	is the space of ordered configurations of $k$ disjoint little disks or little half disks inside the unit half disk
	\[
		\{(x_0,\ldots,x_n) \in \mathbb{R}^{n+1} \mid x_0 \geq 0,\ x_0^2+\ldots+x_n^2 \leq 1\},
	\]
	such that for all $i \in \{1,\ldots,k\}$, we have a little disk if $c_i=\mathsf{f}$ and a half disk if $c_i=\mathsf{h}$. For example, the following picture represents an object in $\mathcal{SC}_1(\mathsf{h},\mathsf{f},\mathsf{f},\mathsf{h},\mathsf{h};\mathsf{h})$:
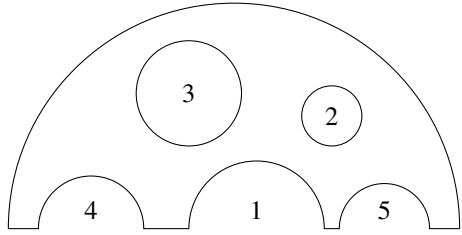
	\[
		\begin{tikzpicture}[scale=1]
			\draw (3,0) -- (2.6,0) arc (0:180:.6) -- (1.2,0) arc (0:180:.9) -- (-1.2,0) arc(0:180:.7) -- (-3,0);
			\draw (.3,0) node[above]{$1$};
			\draw (1.3,1.5) circle (.4) node{$2$};
			\draw (-.6,1.8) circle (.7) node{$3$};
			\draw (-1.9,0) node[above]{$4$};
			\draw (2,0) node[above]{$5$};
			\draw (3,0) arc (0:180:3);
		\end{tikzpicture}
	\]
	The multiplication is given by plugging disks or half disks then reordering.
\end{definition}

\subsection{The relative complete graph operad}

\begin{definition}\cite[Section 3.2]{quesney}
	For $n \geq 1$, the \emph{$n$-th stage filtration $\mathcal{RK}_n$ of the relative complete graph operad} is a Swiss cheese type categorical operad. Its set of colours is again $C := \{\mathsf{f},\mathsf{h}\}$. For $k \geq 0$,
	\[
		\mathcal{RK}_n(\underbrace{\mathsf{f},\ldots,\mathsf{f}}_k;\mathsf{f}) = \mathcal{K}_{n+1}(k)
	\]
	is the poset (and therefore small category) whose
	\begin{itemize}
		\item elements are pairs $(\mu,\sigma)$ with
		\begin{equation}\label{equationmu}
			\mu = (\mu_{ij} \in \{1,\ldots,n+1\})_{1 \leq i < j \leq k}
		\end{equation}
		and $\sigma \in \Sigma_k$,
		\item the poset structure is given by
		\[
			(\mu,\sigma) \leq (\nu,\tau) \Leftrightarrow \text{for all } 1 \leq i < j \leq k, (\mu_{ij},\sigma_{ij})=(\nu_{ij},\tau_{ij}) \text{ or } \mu_{ij}<\nu_{ij},
		\]
		where, for $\sigma \in \Sigma_k$ and $1 \leq i < j \leq k$, $\sigma_{ij} \in \Sigma_2$ is the permutation of $i$ and $j$ by $\sigma$.
	\end{itemize}
	For $(c_1,\ldots,c_k) \in C^k$,
	\[
		\mathcal{RK}_n(c_1,\ldots,c_k;\mathsf{h})
	\]
	is the poset $\mathcal{K}_{n+1}(k)$ with the following extra condition. For $1 \leq i < j \leq k$, $\mu_{ij} \leq n$ if one of the following applies:
	\begin{itemize}
		\item $c_i=c_j=\mathsf{h}$,
		\item $c_i=\mathsf{f}$, $c_j=\mathsf{h}$ and $\sigma_{ij}=id$,
		\item $c_i=\mathsf{h}$, $c_j=\mathsf{f}$ and $\sigma_{ij}=(12)$.
	\end{itemize}
	In order to define the multiplication, note that a collection of integers $\mu$ as in \eqref{equationmu} can be seen as the edge-colouring by numbers $1,\ldots,n+1$ of the complete graph with $k$ vertices. Then $(\mu_1,\sigma_1) \circ_i (\mu_2,\sigma_2) := (\mu_1 \circ_i \mu_2,\sigma_1 \circ_i \sigma_2)$, where $\mu_1 \circ_i \mu_2$ is given by insertion of a graph inside the vertex of another graph then renumbering the vertices, as in the following picture:
	
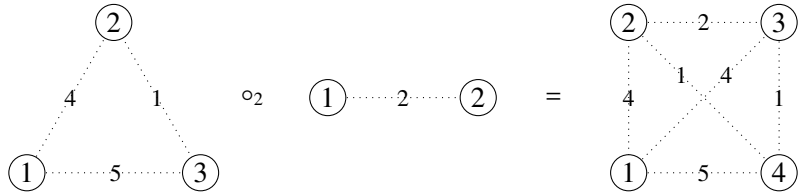
\[
	\begin{tikzpicture}
	\draw[dotted] (0,0) -- (2.309,0) -- (1.155,2) -- (0,0);
	\draw[dotted] (4,1) -- (6,1);
	\draw[dotted] (8,0) -- (10,0);
	\draw[dotted] (10,2) -- (10,0) -- (8,2) -- (8,0) -- (10,2) -- (8,2);
	
	\draw (3,1) node{$\circ_2$};
	\draw (7,1) node{$=$};
	
	\draw[fill=white] (0,0) circle (.24) node{\large{$1$}};
	\draw[fill=white] (2.309,0) circle (.24) node{\large{$3$}};
	\draw[fill=white] (1.155,2) circle (.24) node{\large{$2$}};
	\draw[fill=white] (4,1) circle (.24) node{\large{$1$}};
	\draw[fill=white] (6,1) circle (.24) node{\large{$2$}};
	\draw[fill=white] (8,0) circle (.24) node{\large{$1$}};
	\draw[fill=white] (8,2) circle (.24) node{\large{$2$}};
	\draw[fill=white] (10,2) circle (.24) node{\large{$3$}};
	\draw[fill=white] (10,0) circle (.24) node{\large{$4$}};
	
	\draw (.577,1) node{\small{$4$}};
	\draw (1.18,0) node{\small{$5$}};
	\draw (1.732,1) node{\small{$1$}};
	
	\draw (5,1) node{\small{$2$}};
	
	\draw (9,0) node{\small{$5$}};
	\draw (8,1) node{\small{$4$}};
	\draw (10,1) node{\small{$1$}};
	\draw (8.7,1.3) node{\small{$1$}};
	\draw (9.3,1.3) node{\small{$4$}};
	\draw (9,2) node{\small{$2$}};
	\end{tikzpicture}
	\]
	$\sigma_1 \circ_i \sigma_2$ is the defined the same way as it is defined for the operad $\mathcal{A}ss$ \cite[Example 1.15 (a)]{bergercombinatorial}.
\end{definition}

\begin{theorem}\cite[Theorem 3.5]{quesney}
	The topological operad obtained by applying the classifying space functor pointwise to $\mathcal{RK}_n$ is weakly equivalent to $\mathcal{SC}_n$.
\end{theorem}

\subsection{Categorical Baez-Dolan plus construction}

In this subsection we extend the Baez-Dolan plus construction \cite{baezdolan} to categorical operads.

\begin{definition}\cite[Section 3.1]{baezdolan}
	A \emph{combed tree} is a planar tree equipped with an extra linear order on its set of leaves:
	
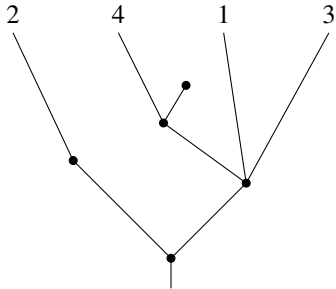
\[
	\begin{tikzpicture}
		\draw (0,-.4) -- (0,0) -- (-1.3,1.3) -- (-2.1,3) node[above]{$2$};
		\draw (0,0) -- (1,1) -- (-.1,1.8) -- (-.7,3) node[above]{$4$};
		\draw (-.1,1.8) -- (.2,2.3);
		\draw (1,1) -- (.7,3) node[above]{$1$};
		\draw (1,1) -- (2.1,3) node[above]{$3$};
		\draw[fill] (0,0) circle(1.5pt);
		\draw[fill] (-1.3,1.3) circle(1.5pt);
		\draw[fill] (1,1) circle(1.5pt);
		\draw[fill] (-.1,1.8) circle(1.5pt);
		\draw[fill] (.2,2.3) circle(1.5pt);
	\end{tikzpicture}
	\]
\end{definition}

\begin{definition}\label{definitioncategoryofoperations}
	Let $\mathcal{P}$ be a categorical operad. The category of \emph{operations} of $\mathcal{P}$ is obtained by taking the coproduct of $\mathcal{P}(c_1,\ldots,c_k;c)$ over all $k \geq 0$ and $(c_1,\ldots,c_k,c) \in C^{k+1}$.
\end{definition}

\begin{definition}\cite[Section 3.3]{baezdolan}
	Let $\mathcal{P}$ be an operad in $\mathrm{Cat}$. A \emph{$\mathcal{P}$-tree} is a combed tree equipped with:
	\begin{itemize}
		\item a colour of $\mathcal{P}$ for each edge,
		\item an operation of $\mathcal{P}$ for each vertex,
	\end{itemize}
	such that if a vertex is decorated with $p \in \mathcal{P}(c_1,\ldots,c_k;c)$, its input edges are decorated with $c_1,\ldots,c_k$ and its output edge with $c$.
\end{definition}

\begin{definition}\label{definitiontarget}
	The \emph{target} of a $\mathcal{P}$-tree is the operation of $\mathcal{P}$ obtained by multiplying the operations which decorate each vertex according to the shape of the tree, and applying the symmetric group actions according to the ordering of the leaves. It is well-defined thanks to the associativity and equivariance axioms for $\mathcal{P}$.
\end{definition}

\begin{definition}\label{definitionpplus}
	For a categorical operad $\mathcal{P}$, let $\mathcal{P}^+$ be the operad in $\mathrm{Set}$ whose:
	\begin{itemize}
		\item colours are operations of $\mathcal{P}$,
		\item for $(c_1,\ldots,c_k,c) \in C^{k+1}$, let
		\[
		\mathcal{P}^+(c_1,\ldots,c_k;c)
		\]
		be the set of combed $\mathcal{P}$-trees with linearly ordered set of vertices $\{v_1,\ldots,v_k\}$, such that the operation which decorates $v_i$ is $c_i$, together with an extra morphism from the target of this $\mathcal{P}$-tree to $c$,
		\item multiplication is given by inserting a $\mathcal{P}$-tree inside the vertex of another $\mathcal{P}$-tree, reordering the set of leaves (combing) and composing the extra morphisms.
	\end{itemize}
\end{definition}

\begin{remark}
	If $\mathcal{P}$ is actually discrete then the extra morphisms must be the identity and we recover the classical definition from \cite{baezdolan}.
\end{remark}

\subsection{Coloured functor-operads}

\begin{definition}\cite[Definition 4.1]{mccluresmith}
	Let $C$ be a set and $A := (A_c)_{c \in C}$ be a $C$-indexed collection of $\mathcal{E}$-categories. A \emph{coloured functor-operad} on $A$ is given by
	\begin{itemize}
		\item for $k \geq 0$ and $(c_1,\ldots,c_k,c) \in C^{k+1}$, an $\mathcal{E}$-functor
		\[
			\xi(c_1,\ldots,c_k;c) : A_{c_1} \otimes \ldots \otimes A_{c_k} \to A_c,
		\]
		\item for $1 \leq i \leq k$, $j \geq 0$, $(c_1,\ldots,c_k,c) \in C^{k+1}$ and $(c'_1,\ldots,c'_j) \in C^j$, an $\mathcal{E}$-natural transformation
		\begin{equation}\label{equationfunctoroperads}
		\xi(c_1,\ldots,c_k;c) \circ_i \xi(c'_1,\ldots,c'_j;c_i) \to \xi(c_1,\ldots,c_{i-1},c'_1,\ldots,c'_j,c_{i+1},\ldots,c_k;c),
		\end{equation}
		where $\circ_i$ is the multiplication of the endomorphism operad $\mathrm{End}_A$ defined in Example \ref{exampleendomorphismoperad},
	\end{itemize}
	satisfying associativity, unitality and equivariance axioms.
\end{definition}

\subsection{Realisation of a coloured operad}

\begin{definition}
	For a coloured operad $\mathcal{P}$ in $\mathcal{E}$, let $u\mathcal{P}$ be the \emph{underlying category} of $\mathcal{P}$. It is the $\mathcal{E}$-category whose objects are colours of $\mathcal{P}$ and morphisms are unary operations, that is
	\[
		u\mathcal{P}(c,d) := \mathcal{P}(c;d).
	\]
\end{definition}

Note that $\mathcal{P}$ induces an $\mathcal{E}$-functor
\begin{equation}
	\mathcal{P}(-,\ldots,-;-): u\mathcal{P}^{op} \otimes \ldots \otimes u\mathcal{P}^{op} \otimes u\mathcal{P} \to \mathcal{E}.
\end{equation}
Also, for two $\mathcal{E}$-categories $\mathcal{U}$ and $\mathcal{V}$, let $\mathcal{V}^\mathcal{U}$ be the $\mathcal{E}$-category of $\mathcal{E}$-functors from $\mathcal{U}$ to $\mathcal{V}$.

\begin{definition}\cite[Section 1.7]{BataninBergerLattice}\label{definitionrealization}
	Let $f: C \to D$ be a function between sets and $\mathcal{P}$ be a $C$-coloured operad. For $d \in D$, let $\mathcal{P}_d$ be the full subcategory of $u\mathcal{P}$ restricted to the objects in $f^{-1}(d)$. For $(d_1,\ldots,d_k,d) \in D^{k+1}$, let
	\begin{equation}\label{equationefunctors}
	\xi_\mathcal{P}^f(d_1,\ldots,d_k;d): \mathcal{E}^{\mathcal{P}_{d_1}} \otimes \ldots \otimes \mathcal{E}^{\mathcal{P}_{d_k}} \to \mathcal{E}^{\mathcal{P}_d}
	\end{equation}
	be the $\mathcal{E}$-functor defined by the coend formula
	\[
	\xi_\mathcal{P}^f(d_1,\ldots,d_k;d)(X_1,\ldots,X_k)(c) := \mathcal{P}(-,\ldots,-;c) \otimes_{\mathcal{P}_{d_1} \otimes \ldots \otimes \mathcal{P}_{d_k}} X_1(-) \otimes \ldots \otimes X_k(-).
	\]
\end{definition}

\begin{proposition}\cite[Proposition 1.8]{BataninBergerLattice}
	The $\mathcal{E}$-functors \eqref{equationefunctors} extend to a coloured functor-operad.
\end{proposition}

\subsection{Condensation of a coloured operad}

For an $\mathcal{E}$-category $\mathcal{U}$, let $\underline{\mathrm{Hom}}_\mathcal{U}$ be the $\mathcal{E}$-valued hom of the $\mathcal{E}$-category $\mathcal{E}^\mathcal{U}$.

\begin{definition}
	Let $f: C \to D$ be a function between sets and $\mathcal{P}$ be a $C$-coloured operad. We also assume that we are given an $\mathcal{E}$-functor $\delta: u\mathcal{P} \to \mathcal{E}$. The \emph{$\delta$-condensation of $\mathcal{P}$ along $f$} is given by
	\[
	\mathrm{Coend}_\mathcal{P}^f(d_1,\ldots,d_k;d) := \underline{\mathrm{Hom}}_{\mathcal{P}_d}(\delta,\xi_\mathcal{P}(d_1,\ldots,d_k;d)(\underbrace{\delta,\ldots,\delta}_k)).
	\]
\end{definition}

\begin{proposition}\cite[Proposition 4.4]{mccluresmith}
	$\mathrm{Coend}_\mathcal{P}^f$ has the structure of a $D$-coloured operad.
\end{proposition}

\begin{proposition}\cite[Proposition 1.5]{BataninBergerLattice}\label{propositionalgebrascondensation}
	Let $f: C \to D$ be a function between sets, $\mathcal{P}$ a $C$-coloured operad and $A$ a $\mathcal{P}$-algebra. Then the $D$-indexed collection
	\[
		\left(\underline{\mathrm{Hom}}_{\mathcal{P}_d}(\delta,A)\right)_{d \in D}
	\]
	is a $\mathrm{Coend}_\mathcal{P}^f$-algebra.
\end{proposition}

\section{Swiss cheese construction for coloured operads}\label{sectionswisscheeseconstruction}

\subsection{Operad for lax pointed categorical algebras}

\begin{definition}
	For a coloured operad $\mathcal{P}$ in $\mathrm{Set}$, let $\mathcal{P}_*$ be the categorical operad for categorical $\mathcal{P}$-algebras $A$ equipped with an \emph{internal algebra} \cite[Definition 7.2]{batanin}, that is a lax morphism of $\mathcal{P}$-algebras $\zeta \to A$, where $\zeta$ is the terminal categorical $\mathcal{P}$-algebra.
\end{definition}

Note that there is a morphism of operads $u: \mathcal{P} \to \mathcal{P}_*$ such that the restriction functor $u^*$ forgets the internal algebra.

\begin{definition}
	A \emph{black and white $\mathcal{P}$-tree} is a $\mathcal{P}$-tree where some vertices are black and some are white, with the extra conditions that a vertex decorated with the unit can not be black and there are no pairs of adjacent black vertices.
\end{definition}

\begin{remark}\label{remarktargetisfunctorial}
	The target (see Definition \ref{definitiontarget}) of a $\mathcal{P}$-tree extends to a functor. Indeed, if $\mathcal{P}$ is a categorical operad, the set of $\mathcal{P}$-trees has a categorical structure. There are morphisms between two $\mathcal{P}$-trees only when they have the same underlying combed tree and the same edge-colouring. Such morphisms are given by a morphism for each vertex. The target is then a functor from the category of $\mathcal{P}$-trees to the category of operations (see Definition \ref{definitioncategoryofoperations}) of $\mathcal{P}$.
\end{remark}

\begin{proposition}\label{propositionpplusstar}
	For a categorical operad $\mathcal{P}$, $\mathcal{P}_*^+ := (\mathcal{P}^+)_*$ is the operad whose:
	\begin{itemize}
		\item colours are operations of $\mathcal{P}$,
		\item for $(c_1,\ldots,c_k,c) \in C^{k+1}$,
		\[
		\mathcal{P}_*^+(c_1,\ldots,c_k;c)
		\]
		is the category whose
		\begin{itemize}
			\item objects are given by black and white $\mathcal{P}$-trees $T$ with linearly ordered set of white vertices $\{v_1,\ldots,v_k\}$, such that the operation of $v_i$ is $c_i$, together with a morphism $\epsilon: t(T) \to c$, where $t$ is the \emph{target functor} of Remark \ref{remarktargetisfunctorial},
			
			\item a morphism $(T,\epsilon) \to (T',\epsilon')$, is given by a morphism for each black vertex, inducing a morphism $\nu: T \to T'$ such that $\epsilon' \cdot t(\nu) = \epsilon$,
		\end{itemize}
		\item multiplication is as in Definition \ref{definitionpplus}.
	\end{itemize}
\end{proposition}

\begin{proof}
	First recall \cite[Definition 9.1]{batanin} that for any $C$-coloured operad $\mathcal{P}$ and categorical $\mathcal{P}$-algebra $A$, an internal algebra in $A$ is given by
	\begin{itemize}
		\item for $c \in C$, an object $a_c \in A_c$,
		
		\item for $k \geq 0$, $(c_1,\ldots,c_k,c) \in C^{k+1}$ and $o \in \mathcal{P}(c_1,\ldots,c_k,c)$, a morphism
		\begin{equation}\label{equationmorphisminternalalgebra}
			m_o(a_{c_1},\ldots,a_{c_k}) \to a_c,
		\end{equation}
		where $m_o$ is the functor coming from the algebra structure,
	\end{itemize}
	satisfying associativity, unitality and equivariance axioms.

	Now let $\mathcal{Q}_*$ be the categorical operad given by the proposition. Let us prove that $\mathcal{Q}_*=\mathcal{P}_*^+$. Let $\mathcal{Q}$ be the suboperad of $\mathcal{Q}_*$ consisting of the operations that are trees containing only white vertices. There is an obvious isomorphism between $\mathcal{Q}$ and $\mathcal{P}^+$.
	
	Now let $A$ be a categorical $\mathcal{Q}$-algebra. Adding to $\mathcal{Q}$ the operations of $\mathcal{Q}_*$ consisting of corollas, whose unique vertex is black, generate $\mathcal{Q}$. So, what remains to prove is that these corollas correspond to an internal algebra in $A$. Let $c$ be a colour of $\mathcal{Q}$, that is an operation of $\mathcal{P}$. The object $a_c \in A_c$ is induced by the corolla with only one black vertex decorated with $c$ and the extra morphism is the identity. Now let $k \geq 0$, $(c_1,\ldots,c_k,c) \in C^{k+1}$ and $o \in \mathcal{Q}(c_1,\ldots,c_k;c)$. So $o$ is a $\mathcal{P}$-tree with $k$ vertices, together with an extra morphism. The morphisms as in \eqref{equationmorphisminternalalgebra} are induced by the extra morphism given by $o$.
\end{proof}

\subsection{Swiss cheese plus construction}

\begin{definition}
	For a set $C$, let $\mathrm{SOp}(C)$ be the operad for $C$-coloured operads.
\end{definition}

Note that $\mathrm{SOp}(C) = \mathrm{Com}(C)^+$, where $\mathrm{Com}(C)$ is the terminal $C$-coloured operad. For a $C$-coloured operad $\mathcal{P}$, there is a canonical map of operad \cite[Section 3.14]{bwdquasitame}
\begin{equation}\label{equationphi}
	\phi: \mathcal{P}^+ \to \mathrm{SOp}(C)
\end{equation}
induced by the unique map $\mathcal{P} \to \mathrm{Com}(C)$.

\begin{definition}\label{definitionalgebraofalgebra}\cite[Definition 3.3]{bwdquasitame}
	Let $\mathcal{P}$ be a categorical operad and $\mathcal{O}$ a categorical $\mathcal{P}^+$-algebra. An \emph{$\mathcal{O}$-algebra} is given by a $C$-indexed collection $A := (A_c)_{c \in C}$ of small categories, together with a morphism of $\mathcal{P}^+$-algebras
	\[
		\mathcal{O} \to \phi^* (\mathrm{End}_A),
	\]
	where $\phi^*$ is the restriction functor induced by \eqref{equationphi}.
\end{definition}

\begin{definition}
	Let $\mathcal{P}_{sc}^+$ be the operad whose algebras are pairs $(\mathcal{O},A)$, where $\mathcal{O}$ is a $\mathcal{P}_*^+$-algebra and $A$ is an $u^*(\mathcal{O})$-algebra.
\end{definition}

\begin{proposition}\label{propositionswisscheeseplus}
	$\mathcal{P}_{sc}^+$ is the Swiss cheese type operad whose:
	\begin{itemize}
		\item set of colours is $B \sqcup C$, where $B$ is the set of operations of $\mathcal{P}$ and $C$ is the set of colours of $\mathcal{P}$,
		\item for $(c_1,\ldots,c_k,c) \in B^{k+1}$,
		\[
			\mathcal{P}_{sc}^+(c_1,\ldots,c_k;c) = \mathcal{P}_*^+(c_1,\ldots,c_k;c).
		\]
		\item for $(c_1,\ldots,c_k) \in (B \sqcup C)^k$ and $c \in C$,
		\[
		\mathcal{P}_{sc}^+(c_1,\ldots,c_k;c)
		\]
		is the category whose objects are black and white $\mathcal{P}$-trees with linearly ordered set of vertices and leaves $\{x_1,\ldots,x_k\}$ such that
		\begin{itemize}
			\item if $c_i \in B$, $x_i$ is a vertex and if $c_i \in C$, $x_i$ is a leaf,
			\item in both cases, $x_i$ is decorated with $c_i$,
			\item the root edge is decorated with $c$,
		\end{itemize}
		and morphisms are given by a morphism for each black vertex,
		\item multiplication is given either by grafting the root of a $\mathcal{P}$-tree to the leaf of another $\mathcal{P}$-tree or by inserting as in Definition \ref{definitionpplus}.
	\end{itemize}
\end{proposition}

\begin{proof}
	Let $\mathcal{Q}$ be the categorical operad given by the proposition. Let us prove that $\mathcal{Q} = \mathcal{P}_{sc}^+$. Since the set of colours of $\mathcal{Q}$ is $B \sqcup C$, a $\mathcal{Q}$-algebra is given by an $B$-indexed collection $\mathcal{O}$ and a $C$-indexed collection $A$. The restriction of $\mathcal{Q}$ to $B$ is $\mathcal{P}_*^+$, so the collection $\mathcal{O}$ has the structure of a $\mathcal{P}_*^+$-algebra. It remains to show that the $\mathcal{Q}$-action on $A$ amounts to an $u^*(\mathcal{O})$-action in the sense of Definition \ref{definitionalgebraofalgebra}. This action is induced by the operations of $\mathcal{Q}$ with target colour in $C$ consisting of corollas with a unique white vertex.
\end{proof}

\subsection{Hochschild object of an algebra over a coloured operad}

\begin{definition}
	For a $C$-coloured operad $\mathcal{P}$, let $\mathcal{P} / \mathrm{SOp}(C)$ be the operad whose algebras are $C$-coloured operads $\mathcal{Q}$ equipped with a map of operads $\mathcal{P} \to \mathcal{Q}$. Let $\mathrm{Gr}(\mathcal{P} / \mathrm{SOp}(C))$ be the operad whose algebras are $C$-coloured operads $\mathcal{Q}$ equipped with a map of operads $\mathcal{P} \to \mathcal{Q}$ and a $\mathcal{Q}$-algebra $A$.
\end{definition}

\begin{lemma}\label{lemmamapofoperads}
	There is a morphism of operads
	\begin{equation}\label{equationfirstmorphism}
		\mathcal{P}_{sc}^+ \to \mathrm{Gr}(\mathcal{P} / \mathrm{SOp}(C)),
	\end{equation}
	which restricts to a morphism
	\begin{equation}\label{equationsecondmorphism}
		\mathcal{P}_*^+ \to \mathcal{P} / \mathrm{SOp}(C).
	\end{equation}
\end{lemma}

\begin{proof}
	Let us describe the Swiss cheese type operad $\mathrm{Gr}(\mathcal{P} / \mathrm{SOp}(C))$ explicitly. We can follow what is done for a slightly different operad in \cite[Section 3.3]{bwdquasitame}. First recall that a \emph{$C$-bouquet} is an $C$-coloured planar corolla. So, it is given by $k \geq 0$ together with $(c_1,\ldots,c_k,c) \in C^{k+1}$:
	
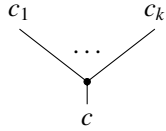
\[
		\begin{tikzpicture}
			\draw[fill] (0,-.3) -- (0,0) circle (1.2pt);
			\draw (0,-.3) node[below]{$c$};
			\draw (-.9,.7) -- (0,0) -- (.9,.7);
			\draw (-.9,.7) node[above]{$c_1$};
			\draw (.9,.7) node[above]{$c_k$};
			\draw (0,.4) node{$\ldots$};
		\end{tikzpicture}
	\]
	The set of colours of $\mathrm{Gr}(\mathcal{P} / \mathrm{SOp}(I))$ is $Bq(C) \sqcup C$, where $Bq(C)$ is the set of $C$-bouquets. The operations are the same as for $\mathcal{P}_{sc}^+$, except that the white vertices are not labelled with an operation of $\mathcal{P}$.
	
	On colours, the morphism \ref{equationfirstmorphism} is given by the function $B \sqcup C \to Bq(C) \sqcup C$ which is the projection $B \to Bq(C)$ and the identity on $C$. On operations, it forgets the decorations of the white vertices. The morphism \ref{equationsecondmorphism} is given by restricting the morphism \ref{equationfirstmorphism} from $B \sqcup C$ to $B$.
\end{proof}

\begin{definition}\label{definitioncategorydeltap}
	For a $C$-coloured categorical operad $\mathcal{P}$, let
	\[
		\Delta \mathcal{P} := u\mathcal{P}_*^+
	\]
	be the underlying category of $\mathcal{P}_*^+$. Recall from Proposition \ref{propositionpplusstar} that the objects of $\Delta \mathcal{P}$ are operations of $\mathcal{P}$. For $c \in C$, let $\Delta \mathcal{P}_c$ be the full subcategory of operation of $\mathcal{P}$ with target $c$.
\end{definition}

\begin{definition}\label{definitionhochschildobject}
	Let $\mathcal{P}$ be a $C$-coloured categorical operad and $A$ a $\mathcal{P}$-algebra. Then $\mathrm{End}_A$ is a $\mathcal{P} / \mathrm{SOp}(C)$-algebra which can be restricted to $\mathcal{P}_*^+$ thanks to the map \eqref{equationsecondmorphism}. Let $\delta: \Delta\mathcal{P} \to \mathrm{Cat}$. For $c \in C$, we define
	\[
	\mathrm{Hoch}(A)_c := \underline{\mathrm{Hom}}_{\Delta \mathcal{P}_c} (\delta,\mathrm{End}_A).
	\]
\end{definition}

\begin{example}
	If $\mathcal{P}$ is the operad $\mathcal{A}ss$ for associative algebras, $\mathcal{P}_*^+ = \mathcal{L}_2$ is the operad for multiplicative non-symmetric operads \cite{BataninBergerLattice}. The underlying category $\Delta \mathcal{P} = u \mathcal{L}_2$ is the simplex category $\Delta$. For an associative algebra $A$ and $\delta$ the canonical inclusion, we recover the classical notion of Hochschild cochain of $A$.
\end{example}

\subsection{Swiss cheese action}

\begin{definition}
	Let $\mathcal{P}$ be a $C$-coloured operad. Let $f: B \sqcup C \to C \sqcup C$ be the function given by the target map $B \to C$ and the identity on $C$. The $C \sqcup C$-coloured operad $\mathrm{SC}(\mathcal{P})$ is defined as the condensation of $\mathcal{P}_{sc}^+$ along $f$.
\end{definition}

\begin{lemma}\label{lemmaalgebra}
	Let $\mathcal{P}$ be a categorical coloured operad and $A$ be a $\mathcal{P}$-algebra. Then the pair $(\mathrm{Hoch}(A),A)$ has the structure of an $\mathrm{SC}(\mathcal{P})$-algebra.
\end{lemma}

\begin{proof}
	By definition of $A$ being a $\mathcal{P}$-algebra, there is a map of operads $\mathcal{P} \to \mathrm{End}_A$. Also note that $A$ is trivially an $\mathrm{End}_A$-algebra. Therefore, using Lemma \ref{lemmamapofoperads}, the pair $(A,\mathrm{End}_A)$ has the structure of a $\mathcal{P}_{sc}^+$-algebra. The conclusion follows from Proposition \ref{propositionalgebrascondensation}.
\end{proof}

The category of small categories is equipped with \emph{Thomason model's structure}. In this model structure, a functor between categories is a weak equivalence when it induces a weak equivalence between nerves.

\begin{conjecture}\label{conjecture}
	Let $\mathcal{P}$ be a coloured categorical operad and $A$ be a $\mathcal{P}$-algebra. Then $\mathrm{Hoch}(A)$ is the terminal object in the homotopy category of Swiss cheese actions (see Definition \ref{definitionswisscheeseaction}) on $A$.
\end{conjecture}

\section{Weak version of Kontsevich's Swiss cheese conjecture}\label{sectionkontsevichconjecture}

\subsection{The complexity map}

Recall \cite[Definition 4.1]{weberoperads} that one can define a notion of a \emph{lax morphism between categorical operads}. Briefly, for a lax morphism of operads $f: \mathcal{P} \to \mathcal{Q}$, instead of having $f(x \circ_i y) = f(x) \circ_i f(y)$ as is the case for classical morphisms of operads, we assume that $f$ comes equipped with natural transformations $f(x \circ_i y) \to f(x) \circ_i f(y)$ satisfying obvious axioms.

\begin{lemma}\label{lemmaexistencecomplexitymap}
	There is a lax morphism of operads
	\begin{equation}\label{equationmapc}
		f: (\mathcal{K}_n)_{sc}^+ \to \mathcal{RK}_{n+1}.
	\end{equation}
\end{lemma}

\begin{proof}
	Recall from Proposition \ref{propositionswisscheeseplus} that the set of colours of $(\mathcal{K}_n)_{sc}^+$ is $B \sqcup 1$, where $\mathcal{B}$ is the set of operations of $\mathcal{K}_n$ and $1 = \{*\}$ is the singleton set. The set of operations of $\mathcal{RK}_{n+1}$ is $\{\mathsf{f},\mathsf{h}\}$. On colours, $f$ sends an element of $B$ to $\mathsf{f}$ and $*$ to $\mathsf{h}$.
	
	Now let $(c_1,\ldots,c_k,c) \in (B \sqcup 1)^{k+1}$ and let us define a functor
	\[
	f: (\mathcal{K}_n)_{sc}^+(c_1,\ldots,c_k;c) \to \mathcal{RK}_{n+1}(f(c_1),\ldots,f(c_k);f(c))
	\]
	Let $x \in (\mathcal{K}_n)_{sc}^+(c_1,\ldots,c_k;c)$. Let $f(x) := (\mu,\sigma)$, where $\mu$ and $\sigma$ are defined as follows. Recall that $x$ is a black and white $\mathcal{K}_n$-tree and each element $\{1,\ldots,k\}$ corresponds to a vertex or a leaf of this tree. For $1 \leq i < j \leq k$, if $i$ and $j$ are one above the other, then $\mu_{ij}:=n+1$ and
	\[
	\sigma_{ij} = 
	\begin{cases}
	id &\text{if $i$ is above $j$,} \\
	(12) &\text{if $i$ is below $j$.}
	\end{cases}
	\]
	If $i$ and $j$ are not one above the other, then there is a maximal vertex $v$ in the $\mathcal{K}_n$-tree $x$ which lies below $i$ and $j$. Note that $v$ is decorated with an operation $(\mu',\sigma') \in \mathcal{K}_n(|v|)$, where $|v|$ is the number of input edges of $v$. Let $l$ and $m$ in $\{1,\ldots,|v|\}$ corresponding to the edges that lie below $i$ and $j$ respectively. We define $(\mu_{ij},\sigma_{ij}) := (\mu'_{lm},\sigma'_{lm})$. As an illustration, the functor
	\begin{equation}\label{equationexamplecomplexitymap}
		f: (\mathcal{K}_2)_{sc}^+(*,*,c_3,c_4,*;*) \to \mathcal{RK}_3(\mathsf{h},\mathsf{h},\mathsf{f},\mathsf{f},\mathsf{h};\mathsf{h})
	\end{equation}
	can be pictured as follows:
	
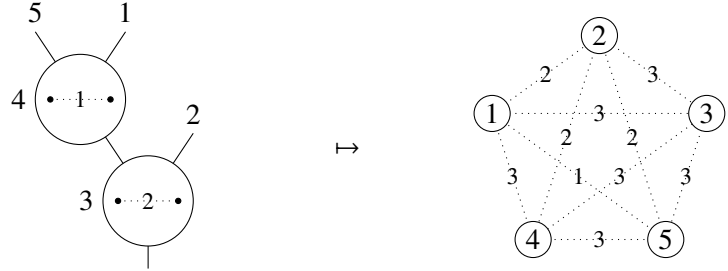
\[
	\begin{tikzpicture}
	\draw (0,-.9) -- (0,0) -- (-1.5,2.25) node[above]{\large{$5$}};
	\draw (0,0) -- (.6,.9) node[above]{\large{$2$}};
	\draw (-.9,1.35) -- (-.3,2.25) node[above]{\large{$1$}};
	\draw (-.6,0) node[left]{\large{$3$}};
	\draw (-1.5,1.35) node[left]{\large{$4$}};
	
	\draw[fill=white] (0,0) circle (.6);
	\draw[dotted] (-.4,0) -- (.4,0);
	\draw (0,0) node{\small{$2$}};
	\draw[fill] (-.4,0) circle (1pt);
	\draw[fill] (.4,0) circle (1pt);
	\draw[fill=white] (-.9,1.35) circle (.6);
	\draw[dotted] (-1.3,1.35) -- (-.5,1.35);
	\draw (-.9,1.35) node{\small{$1$}};
	\draw[fill] (-1.3,1.35) circle (1pt);
	\draw[fill] (-.5,1.35) circle (1pt);
	
	\draw (2.65,.7) node{$\mapsto$};
	
	\begin{scope}[shift={(6,.7)},scale=1.5]
	\draw[dotted] ({cos(54)},{-sin(54)}) -- ({cos(18)},{sin(18)}) -- (0,1) -- ({-cos(18)},{sin(18)}) -- ({-cos(54)},{-sin(54)}) -- ({cos(54)},{-sin(54)}) -- (0,1) -- ({-cos(54)},{-sin(54)}) -- ({cos(18)},{sin(18)}) -- ({-cos(18)},{sin(18)}) -- ({cos(54)},{-sin(54)});
	
	\draw[fill=white] ({cos(18)},{sin(18)}) circle (.16) node{\large{$3$}};
	\draw[fill=white] (0,1) circle (.16) node{\large{$2$}};
	\draw[fill=white] ({-cos(18)},{sin(18)}) circle (.16) node{\large{$1$}};
	\draw[fill=white] ({-cos(54)},{-sin(54)}) circle (.16) node{\large{$4$}};
	\draw[fill=white] ({cos(54)},{-sin(54)}) circle (.16) node{\large{$5$}};
	
	\draw (0,{-sin(54)}) node{\small{$3$}};
	\draw ({cos(18)/2+cos(54)/2},{sin(18)/2-sin(54)/2}) node{\small{$3$}};
	\draw ({cos(18)/2},{sin(18)/2+1/2}) node{\small{$3$}};
	\draw ({-cos(18)/2},{sin(18)/2+1/2}) node{\small{$2$}};
	\draw ({-cos(18)/2-cos(54)/2},{sin(18)/2-sin(54)/2}) node{\small{$3$}};
	\draw ({cos(18)/2-cos(54)/2},{sin(18)/2-sin(54)/2}) node{\small{$3$}};
	\draw (0,{sin(18)}) node{\small{$3$}};
	\draw ({cos(54)/2-cos(18)/2},{sin(18)/2-sin(54)/2}) node{\small{$1$}};
	\draw ({cos(54)/2},{1/2-sin(54)/2}) node{\small{$2$}};
	\draw ({-cos(54)/2},{1/2-sin(54)/2}) node {\small{$2$}};
	\end{scope}
	\end{tikzpicture}
	\]
	$c_3$ and $c_4$ in \eqref{equationexamplecomplexitymap} are the operations which decorate the vertices numbered by $3$ and $4$ respectively.
	
	Note that if $i$ and $j$ satisfy one of the following:
	\begin{itemize}
		\item $f(c_i)=f(c_j)=\mathsf{h}$,
		\item $f(c_i)=\mathsf{f}$, $f(c_j)=\mathsf{h}$ and $\sigma_{ij}=id$,
		\item $f(c_i)=\mathsf{h}$, $f(c_j)=\mathsf{f}$ and $\sigma_{ij}=(12)$,
	\end{itemize}
	then, since nothing can be above a leaf, $i$ and $j$ can not be one above the other, so $\mu_{ij} \leq n$. This proves that $f$ is well-defined. The proof that $f$ is indeed a lax morphism of operads is left to the reader.
\end{proof}

\subsection{A formula for the coend}

\begin{definition}\cite[Definition 3.1]{delegergregocondensation}
	Let $\mathcal{P}$ be a coloured operad in $\mathrm{Cat}$ with set of colours $C$ and $f: C \to D$ be a function of sets. For $c \in C$ and $\mathbf{d} := (d_1,\ldots,d_k) \in D^k$, we define $\mathcal{P}_{\mathbf{d}}^f/c$ as the category whose
	\begin{itemize}
		\item objects are elements $\mathbf{c}=(c_1,\ldots,c_k) \in C^k$ such that $f^k(\mathbf{c})=\mathbf{d}$, together with an operation $o \in \mathcal{P}(c_1,\ldots,c_k;c)$,
		\item morphisms from $o \in \mathcal{P}(c_1,\ldots,c_k;c)$ to $o' \in \mathcal{P}(c'_1,\ldots,c'_k;c)$ are given by $k$-tuples of unary operations
		\begin{equation}\label{equationktuple}
		(p_1,\ldots,p_k) \in \mathcal{P}(c_1;c'_1) \times \ldots \times \mathcal{P}(c_k;c'_k)
		\end{equation}
		together with an extra morphism
		\begin{equation}\label{equationextramorphism}
			o \to m(o',p_1,\ldots,p_k).
		\end{equation}
	\end{itemize}
\end{definition}

\begin{remark}
	The morphisms of $\mathcal{P}_{\mathbf{d}}^f/c$ are generated by two types of morphisms. The first type is when the extra morphism \eqref{equationextramorphism} is the identity. The second type is when the $k$-tuple \eqref{equationktuple} is given by the units of $\mathcal{P}$.
\end{remark}

\begin{example}\label{exampleformula}
	Let $\mathcal{P}=(\mathcal{K}_n)_{sc}^+$, $B$ the set of operations of $\mathcal{K}_n$ and $f: B \sqcup 1 \to \{\mathsf{f},\mathsf{h}\}$ as defined in Lemma \ref{lemmaexistencecomplexitymap}. For $c \in B \sqcup 1$ and $\mathbf{d}=(d_1,\ldots,d_k) \in \{\mathsf{f},\mathsf{h}\}^k$, let $l$ be the number of elements $i \in \{1,\ldots,k\}$ such that $d_i=\mathsf{f}$ and $m$ the number of elements $i \in \{1,\ldots,k\}$ such that $d_i=\mathsf{h}$. Of course, $l+m=k$. The objects of $\mathcal{P}_{\mathbf{d}}^f/c$ are black and white $\mathcal{K}_n$-trees with exactly $m$ white vertices. If $f(c)=\mathsf{f}$, there is an extra morphism from the target of this tree to $c$. Note that this extra morphism is unambiguously given by the black and white $\mathcal{K}_n$-tree since the category of operations of $\mathcal{K}_n$ is a poset. If $f(c)=\mathsf{h}$, the tree must have exactly $n$ leaves. The first type of morphisms of $\mathcal{P}_{\mathbf{d}}^f/c$ are contractions of edges between a black and a white vertex. The second type are morphisms between black and white $\mathcal{K}_n$-trees which have exactly the same shape. They are given by morphisms in the category of operations of $\mathcal{K}_n$ for each black vertex.
	
	For example, assume that $c$ is the unique nullary operation of $\mathcal{K}_2$ and $\mathbf{d}=(\mathsf{f},\mathsf{f})$. Then $\mathcal{P}_{\mathbf{d}}^f/c$ contains the subcategory drawn below:
	
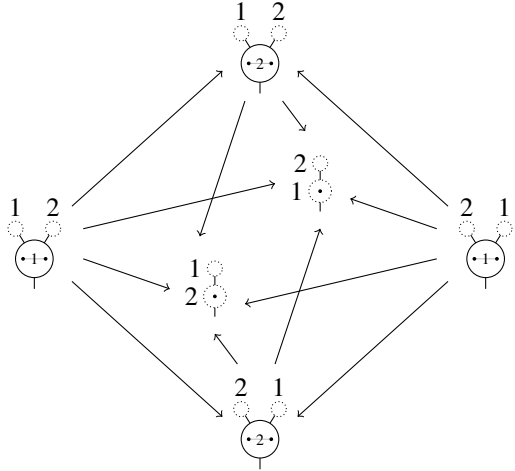
\[
		\begin{tikzpicture}
			\begin{scope}[shift={(-3,0)}]
				\draw (0,-.6) -- (0,-.2);
				\draw (-.25,.2) -- (0,-.2) -- (.25,.2);
				\draw[fill=white,densely dotted] (-.25,.2) circle (.1);
				\draw (-.25,.25) node[above]{$1$};
				\draw (.25,.25) node[above]{$2$};
				\draw[fill=white,densely dotted] (.25,.2) circle (.1);
				\draw[fill=white] (0,-.2) circle (.25);
				\draw[nearly transparent] (-.15,-.2) -- (.15,-.2);
				\draw[fill] (-.15,-.2) circle (.5pt);
				\draw[fill] (.15,-.2) circle (.5pt);
				\draw (0,-.2) node{\tiny{$1$}};
			\end{scope}
			
			\begin{scope}[shift={(3,0)}]
				\draw (0,-.6) -- (0,-.2);
				\draw (-.25,.2) -- (0,-.2) -- (.25,.2);
				\draw[fill=white,densely dotted] (-.25,.2) circle (.1);
				\draw (-.25,.25) node[above]{$2$};
				\draw (.25,.25) node[above]{$1$};
				\draw[fill=white,densely dotted] (.25,.2) circle (.1);
				\draw[fill=white] (0,-.2) circle (.25);
				\draw[nearly transparent] (-.15,-.2) -- (.15,-.2);
				\draw[fill] (-.15,-.2) circle (.5pt);
				\draw[fill] (.15,-.2) circle (.5pt);
				\draw (0,-.2) node{\tiny{$1$}};
			\end{scope}
			
			\begin{scope}[shift={(0,-2.4)}]
				\draw (0,-.6) -- (0,-.2);
				\draw (-.25,.2) -- (0,-.2) -- (.25,.2);
				\draw[fill=white,densely dotted] (-.25,.2) circle (.1);
				\draw (-.25,.25) node[above]{$2$};
				\draw (.25,.25) node[above]{$1$};
				\draw[fill=white,densely dotted] (.25,.2) circle (.1);
				\draw[fill=white] (0,-.2) circle (.25);
				\draw[nearly transparent] (-.15,-.2) -- (.15,-.2);
				\draw[fill] (-.15,-.2) circle (.5pt);
				\draw[fill] (.15,-.2) circle (.5pt);
				\draw (0,-.2) node{\tiny{$2$}};
			\end{scope}
			
			\begin{scope}[shift={(0,2.6)}]
				\draw (0,-.6) -- (0,-.2);
				\draw (-.25,.2) -- (0,-.2) -- (.25,.2);
				\draw[fill=white,densely dotted] (-.25,.2) circle (.1);
				\draw (-.25,.25) node[above]{$1$};
				\draw (.25,.25) node[above]{$2$};
				\draw[fill=white,densely dotted] (.25,.2) circle (.1);
				\draw[fill=white] (0,-.2) circle (.25);
				\draw[nearly transparent] (-.15,-.2) -- (.15,-.2);
				\draw[fill] (-.15,-.2) circle (.5pt);
				\draw[fill] (.15,-.2) circle (.5pt);
				\draw (0,-.2) node{\tiny{$2$}};
			\end{scope}
			
			\begin{scope}[shift={(.8,.7)}]
				\draw (0,-.27) -- (0,.37);
				\draw[fill=white,densely dotted] (0,.37) circle (.1);
				\draw[fill=white,densely dotted] (0,0) circle (.15);
				\draw[fill] (0,0) circle (.5pt);
				\draw (-.1,0) node[left]{$1$};
				\draw (-.05,.37) node[left]{$2$};
			\end{scope}

			\begin{scope}[shift={(-.6,-.7)}]
				\draw (0,-.27) -- (0,.37);
				\draw[fill=white,densely dotted] (0,.37) circle (.1);
				\draw[fill=white,densely dotted] (0,0) circle (.15);
				\draw[fill] (0,0) circle (.5pt);
				\draw (-.1,0) node[left]{$2$};
				\draw (-.05,.37) node[left]{$1$};
			\end{scope}
			
			\draw[->] (-2.5,.5) -- (-.5,2.3);
			\draw[->] (-2.5,-.5) -- (-.5,-2.3);
			\draw[->] (2.5,.5) -- (.5,2.3);
			\draw[->] (2.5,-.5) -- (.5,-2.3);
			\draw[->] (-2.35,-.2) -- (-1.2,-.6);
			\draw[->] (-2.35,.2) -- (.2,.8);
			\draw[->] (2.35,.2) -- (1.2,.6);
			\draw[->] (2.35,-.2) -- (-.2,-.8);
			\draw[->] (-.3,-1.6) -- (-.6,-1.2);
			\draw[->] (.2,-1.6) -- (.8,.2);
			\draw[->] (.3,1.9) -- (.6,1.5);
			\draw[->] (-.2,1.9) -- (-.8,.1);
		\end{tikzpicture}
	\]
	In the picture above, the white vertices are drawn as dotted circles, while the black vertices are drawn as plain circles. Note that we get a $2$-dimensional sphere. On the other hand, $\mathcal{RK}_2(\mathsf{f},\mathsf{f};\mathsf{f})$ is also equivalent to a $2$-dimensional sphere. In fact, the complexity map $f$ induces a weak equivalence $\mathcal{P}_{\mathbf{d}}^f/c \to \mathcal{RK}_d(d_1,\ldots,d_k;f(c))$ in general. This will be proved in Lemma \ref{lemmatechnical}.
\end{example}

\begin{lemma}\label{lemmaformulacoend}
	Let $\mathcal{P}$ be a categorical operad with set of colours $C$, $f: C \to D$ a function, $c \in C$, $d:=f(c)$ and $\mathbf{d} = (d_1,\ldots,d_k) \in D^k$. Let $\delta: u\mathcal{P} \to \mathrm{Cat}$ be the functor sending $c$ to the comma category $u\mathcal{P}/c$. Then there is an isomorphism
	\[
		\xi_{\mathcal{P}}^f(d_1,\ldots,d_k;d)(\delta,\ldots,\delta)(c) \simeq \mathcal{P}_{\mathbf{d}}^f/c.
	\]
\end{lemma}

\begin{proof}
	We can apply the same arguments as in the proof of \cite[Lemma 3.3]{delegergregocondensation}.
\end{proof}

\subsection{Weak equivalence induced by the complexity map}

\begin{lemma}\label{lemmainducedlaxmorphism}
	A lax morphism of operads $f: \mathcal{P} \to \mathcal{Q}$ induces a lax morphism of operads $\mathrm{Coend}_\mathcal{P}^f \to \mathrm{Coend}_\mathcal{Q}^{id} \simeq \mathcal{Q}$.
\end{lemma}

\begin{proof}
	Let $(d_1,\ldots,d_k,d) \in D^{k+1}$. For $(c_1,\ldots,c_k,c) \in \mathcal{P}_{d_1} \times \ldots \times \mathcal{P}_{d_k} \times \mathcal{P}_d$, there is a composite
	\begin{equation}\label{equationfunctors}
		\mathcal{P}(c_1,\dots,c_k;c) \times \delta(c_1) \times \ldots \times \delta(c_k) \to \mathcal{P}(c_1,\dots,c_k;c) \to \mathcal{Q}(d_1,\ldots,d_k;d),
	\end{equation}
	where the first functor is the projection and the second is given by $f$. Abusing notations, let $\mathcal{Q}(d_1,\ldots,d_k;d)$ be the constant functor $\mathcal{P}_d \to \mathrm{Cat}$ sending everything to $\mathcal{Q}(d_1,\ldots,d_k;d)$. By the universal property of coends, the functors \eqref{equationfunctors} induce a natural transformation
	\begin{equation}\label{equationnaturaltransformations}
		\xi_\mathcal{P}^f(d_1,\ldots,d_k;d) := \xi_\mathcal{P}^f(d_1,\ldots,d_k;d)(\delta,\ldots,\delta) \Rightarrow \mathcal{Q}(d_1,\ldots,d_k;d).
	\end{equation}
	Recall that, for a category $\mathcal{C}$, two functors $F,G: \mathcal{C} \to \mathrm{Cat}$ and two natural transformations $\alpha,\beta: F \Rightarrow G$, a \emph{modification} $\tau: \alpha \Rrightarrow \beta$ is given by a natural transformation $\tau_c: \alpha_c \Rightarrow \beta_c$ for all $c \in \mathcal{C}$, such that for all $h:c \to d$ in $\mathcal{C}$, the following natural transformations coincide:
	\[
		\xymatrix{
			F(c) \ar[d]_{F(h)} \rtwocell^{\alpha_c}_{\beta_c}{\tau_c} & G(c) \ar[d]^{G(h)} && F(c) \ar[d]_{F(h)} \ar@/^/[r]^{\alpha_c} & G(c) \ar[d]^{G(h)} \\
			F(d) \ar@/_/[r]_{\beta_d} & G(d) && F(d) \rtwocell^{\alpha_d}_{\beta_d}{\tau_c} & G(d)
		}
	\]
	Let $1 \leq i \leq k$, $j \geq 0$, $(d_1,\ldots,d_k,d) \in D^{k+1}$ and $(d'_1,\ldots,d'_j) \in D^j$. The natural transformations \eqref{equationnaturaltransformations} give natural transformations
	\[
		\alpha: \xi_\mathcal{P}^f(d_1,\ldots,d_k;d) \circ_i \xi_\mathcal{P}^f(d'_1,\ldots,d'_j;d_i) \to \mathcal{Q}(d_1,\ldots,d_k;d) \circ_i \mathcal{Q}(d'_1,\ldots,d'_j;d_i)
	\]
	and
	\[
	\beta: \xi_\mathcal{P}^f(d_1,\ldots,d_k;d) \circ_i \xi_\mathcal{P}^f(d'_1,\ldots,d'_j;d_i)\to \mathcal{Q}(d_1,\ldots,d_{i-1},d'_1,\ldots,d'_j,d_{i+1},\ldots,d_k;d).
	\]
	$\beta$ is obtained using the natural transformation \eqref{equationfunctoroperads} coming from the functor-operad structure. $\alpha$ and $\beta$ come equipped with a modification $\tau: \alpha \Rrightarrow \beta$ which is induced by the lax structure of $f$.
	
	Finally, the functors
	\begin{equation}\label{equationinducedfunctorcoend}
		\mathrm{Coend}_\mathcal{P}^f(d_1,\ldots,d_k;d) \to \mathcal{Q}(d_1,\ldots,d_k;d)
	\end{equation}
	are induced by the natural transformations \eqref{equationnaturaltransformations} and the lax structure is induced by the modifications.
\end{proof}

By definition, a lax morphism of operads is a \emph{weak equivalence} if the underlying functors are weak equivalences.

\begin{lemma}\label{lemmatechnical}
	The lax morphism of operad
	\begin{equation}\label{equationimportant}
		\mathrm{SC}(\mathcal{K}_n) \to \mathcal{RK}_n
	\end{equation}
	induced by the complexity map \eqref{equationmapc} and given by Lemma \ref{lemmainducedlaxmorphism} is a weak equivalence.
\end{lemma}

\begin{proof}
	To simplify the notations, let $\mathcal{P}=(\mathcal{K}_n)_{sc}^+$ and $\mathcal{Q}=\mathcal{RK}_n$. Let $c \in B \sqcup 1$, where $B$ is the set of operations of $\mathcal{K}_n$, $d := f(c)$ and $\mathbf{d} = (d_1,\ldots,d_k) \in \{\mathsf{f},\mathsf{h}\}^k$. The complexity map $f$ induces a functor
	\begin{equation}\label{equationinducedfunctor}
		\mathcal{P}_\mathbf{d}^f/c \to \mathcal{Q}_\mathbf{d}^{id}/d
	\end{equation}
	which will be denoted by $f$ again, abusing notations. This gives us natural transformations which, according to Lemma \ref{lemmaformulacoend}, are exactly the natural transformations \eqref{equationnaturaltransformations}. Note that if these natural transformations are pointwise weak equivalences, the desired lax morphism of operads \eqref{equationimportant} induced by them is also a weak equivalence. So we will get the desired result if we prove that the functor $f$ of \eqref{equationinducedfunctor} is a weak equivalence for all $c$ and $\mathbf{d}$.
	
	Recall that a functor $F: \mathcal{X} \to \mathcal{Y}$ is \emph{homotopy left cofinal} if for all $y \in \mathcal{Y}$, the \emph{slice category} $F/y$ is contractible \cite[Definition 19.6.1]{hirschhorn}. According to \cite[Theorem A]{quillenthma}, if a functor is homotopy left cofinal, then it is a weak equivalence. So we have to prove that the functor $f$ is homotopy left cofinal.
	
	Note that $\mathcal{Q}_\mathbf{d}^{id}/d$ is simply the category $\mathcal{RK}_n(d_1,\ldots,d_k;d)$. Let $q = (\mu,\sigma) \in \mathcal{Q}_\mathbf{d}^{id}/d$. Let us prove that $f/q$ is contractible. We will follow the arguments of \cite[Lemma 3.4]{delegergregocondensation}. We proceed by induction on $k$.
	
	First let us assume that $k=0$. Then $q$ is the unique nullary operation. Since $\mathcal{Q}_\mathbf{d}^{id}/d = \mathcal{RK}_n(;d)$ is the trivial category, $f/q$ is isomorphic to $\mathcal{P}_\mathbf{d}^f/c$. This category was described in Example \ref{exampleformula}. When $k=0$, its objects can only be corollas whose unique vertex is black. It has a terminal object given by the corolla whose unique vertex is decorated by $c$ if $c \in B$ and the unique nullary operation if $c = *$. Therefore, it is contractible.
	
	Now let us assume that $k \geq 1$. Let $i := \sigma^{-1}(1)$ and $\mathbf{d}'=(d_1,\ldots,d_{i-1},d_{i+1},\ldots,d_k)$. Since $f$ is a lax morphism of operads, we have a natural transformation
	\begin{equation}\label{squareinduction}
		\xymatrix{
			\mathcal{P}_\mathbf{d}^f/c \ar@{}[dr]|{\rotatebox{45}{$\Rightarrow$}} \ar[r]^f \ar[d]_{- \circ_i o} & \mathcal{Q}_\mathbf{d}^{id}/d \ar[d]^{- \circ_i c(o)} \\
			\mathcal{P}_{\mathbf{d}'}^f/c \ar[r]_f & \mathcal{Q}_{\mathbf{d}'}^{id}/d
		}
	\end{equation}
	where $o$ is the unique nullary operation of $\mathcal{P}$ with target $c_i$. Let $q' := q \circ_i f(o)$. The square \eqref{squareinduction} induces a functor
	\[
		F: f/q \to f/q'
	\]
	sending $f(p) \to q$ to the composite
	\[
		f(p \circ_i o) \to f(p) \circ_i f(o) \to q \circ_i f(o) = q'.
	\]
	Concretely, if $d_i=\mathsf{f}$, $F$ turns the white vertex numbered by $i$ to a black vertex and, if necessary, contracts edges between black vertices and removes unary black vertices. If $d_i=\mathsf{h}$, $F$ removes the leaf edge numbered by $i$ and, if necessary, composes the operation which decorates the vertex below this leaf edge with the nullary operation. Note that since $i = \sigma^{-1}(1)$, the white vertex or leaf numbered by $i$ can not be above another white vertex.
	
	We will now prove that $F$ is homotopy left cofinal. Let $x \in f/q'$, given by $p' \in \mathcal{P}_{\mathbf{d}'}^f/c$ together with a map $f(p') \to q'$. We want to prove that $F/x$ is contractible. As in \cite[Section 5.3]{cisinski}, we consider the canonical inclusion functor
	\begin{equation}\label{canonicalinclusionfunctor}
		F_x \to F/x,
	\end{equation}
	where $F_x$ is the fibre of $F$ over $x$.
	
	Let us prove that $F_x$ has a terminal object. We assume without loss of generality that the root vertex of $p'$ is black. Indeed, we can always insert a unary black vertex to the root edge. Let $r$ be the operation which decorates this root vertex. Then $r = (\mu',\sigma') \in \mathcal{K}_n(l)$, where $l$ is the number of input edges of the root vertex. Each element of $\{1,\ldots,l\}$ corresponds to a vertex of the complete graph $r$, as well as an input edge of the root vertex. Let $S \subset \{1,\ldots,l\}$ defined as follows. $s \in S$ if for all $j \in \{1,\ldots,k\}$ such that the white vertex or leaf numbered by $j$ is above $s$, $\mu_{ij} \leq n$. Let $t$ be the following operation of $\mathcal{K}_n$. It is given by a complete graph on the set $S \sqcup 1$. The edge-colouring on $S$ is given by the restriction of $r$ to $S$. Let $\overline{S}$ be the complement of $S$ in $\{1,\ldots,l\}$. The edge between $i \in S$ and $*$ is coloured with the minimum over all $j \in \overline{S}$ of $\mu'_{ij}$. Let $t'$ be the operation of $\mathcal{K}_n$ given by restriction of $r$ to $\overline{S}$. Let $r'$ be the following black and white $\mathcal{K}_n$-tree. It has exactly two vertices, one black and one white. The root vertex is black and decorated with $t$. The white vertex is above the incoming edge of the root vertex which corresponds to $*$. It is decorated with $t'$. Now let $p$ be the black and white $\mathcal{K}_n$-tree obtained by inserting $r'$ inside the root vertex of $p'$. Then $p$ is the terminal object of $F_x$.
	
	Finally, let us prove that the $F_x$ is a reflective subcategory of $F/x$. That is, we want to prove that the functor \eqref{canonicalinclusionfunctor} has a left adjoint $L$. Let $F(y) \to x$ be an object in $F/x$. Let $v$ be the black vertex of $F(y)$ which takes the place of the $i$-th white vertex which was in $y$. We get a uniquely defined black vertex $w$ of $x$ as the image of $v$ through the map $F(y) \to x$. If the map $F(y) \to x$ contracts an edge between $v$ and a white vertex, $w$ is given by inserting a unary black vertex. We define $L(y)$ as the black vertex $w$ in $x$ back to a white vertex. The unit of the adjunction is given by the map $F(y) \to x$. The counit is the identity.
	
	In conclusion, the functor \eqref{canonicalinclusionfunctor} is a right adjoint, so it is a weak equivalence. Since, for all $x \in f/q'$, $F_x$ is contractible, so is $F/x$. We deduce that $F$ is homotopy left cofinal. Therefore it is a weak equivalence. By induction, $f/q$ is contractible for all $q \in \mathcal{Q}_\mathbf{d}^{id}/d$. So $f$ is also homotopy left cofinal.
\end{proof}

\subsection{Proof of the main theorem}

\begin{lemma}\label{lemmazigzagwe}
	Let $\mathcal{P}$ and $\mathcal{Q}$ be two categorical operads. If there is a lax morphism of operads $f: \mathcal{P} \to \mathcal{Q}$ which is a weak equivalence, then $\mathcal{P}$ and $\mathcal{Q}$ are (strictly) weakly equivalent.
\end{lemma}

\begin{proof}
	Let $C$ be the set of colours of $\mathcal{P}$. Let $\hat{\mathcal{P}}$ be the following $C$-coloured operad. For $(c_1,\ldots,c_k,c) \in C^{k+1}$, we write
	\[
		f: \mathcal{P}(c_1,\ldots,c_k;c) \to \mathcal{Q}(f(c_1),\ldots,f(c_k);f(c))
	\]
	again for the functor given as part of the structure of the lax morphism $f$. Let
	\[
		\hat{\mathcal{P}}(c_1,\ldots,c_k;c)
	\]
	be the comma category $f/id$, as defined in \cite[Section II.6]{maclane}. Explicitly, it is the category whose
	\begin{itemize}
		\item objects are triples $(p,q,\alpha)$, where $p \in \mathcal{P}(c_1,\ldots,c_k;c)$, $q \in \mathcal{Q}(f(c_1),\ldots,f(c_k);f(c))$ and $\alpha: f(p) \to q$ is a morphism in $\mathcal{Q}(f(c_1),\ldots,f(c_k);f(c))$,
		
		\item a morphism $(p,q,\alpha) \to (p',q',\alpha')$ is given by morphisms $\phi: p \to p'$ and $\psi: q \to q'$ such that $\psi \alpha = \alpha' f(\phi)$.
	\end{itemize}
	It is immediate to see that $\hat{\mathcal{P}}$ has the structure of an operad. Moreover, there are obvious (strict) morphisms of operads
	\[
		\mathcal{P} \xleftarrow{\pi} \hat{\mathcal{P}} \xrightarrow{\psi} \mathcal{Q}
	\]
	given by the projections. $\pi$ has a pointwise left adjoint $\lambda$, which sends $p$ to $(p,f(p),id)$. In particular, it is a weak equivalence. Since $\psi \lambda = f$, $\psi$ is also a weak equivalence.
\end{proof}

\begin{theorem}\label{theoremswisscheese}
	The categorical operads $\mathrm{SC}(\mathcal{K}_n)$ and $\mathcal{RK}_n$ are weakly equivalent.
\end{theorem}

\begin{proof}
	It follows immediately from Lemma \ref{lemmatechnical} and Lemma \ref{lemmazigzagwe}.
\end{proof}

\begin{corollary}\label{corollaswisscheeseaction}
	If $A$ is a $\mathcal{K}_n$-algebra, then the pair $(A,\mathrm{Hoch}(A))$ is an $\mathcal{RK}_n$-algebra.
\end{corollary}

\begin{proof}
	It follows immediately from Lemma \ref{lemmaalgebra} and Theorem \ref{theoremswisscheese}.
\end{proof}

\section{Action on the Hochschild-Pirashvili cochains}\label{sectionhochschildpirashvili}

\subsection{Simplicial spheres}

\begin{definition}
	For $n \geq 1$, let $S^n := \Delta^n / \partial \Delta^n$ be the \emph{$n$-dimensional simplicial sphere}.
\end{definition}

\begin{remark}\label{remarksimplicialsphere}
	There is a canonical bijection \cite[page 35]{BataninBergerLattice}
	\[
	S^n(m) \simeq {m \choose n} \sqcup 1,
	\]
	where ${m \choose n}$ is the set of subsets of $\{1,\ldots,m\}$ which have $n$ elements and $1 = \{*\}$ is the singleton set. Indeed, $\Delta^n(m)$ is the set of order-preserving functions from $m$ to $n$. The functions in the interior of $\Delta^n(m)$, that is the function which are not in $\partial \Delta^n(m)$, are the surjective ones. By Joyal duality, they correspond to order-preserving injections from $\{1,\ldots,n\}$ to $\{1,\ldots,m\}$.
	
	Moreover, the morphisms of $\Delta$ correspond to planar trees with black vertices and a unique white vertex, such that there are no unary black vertices or edges between two black vertices. The source is given by the number of edges above the white vertex and the target is given by the number of leaves. For example, the planar tree corresponding to the map $f: 3 \to 5$ given by
	\begin{equation}\label{equationmorphismindelta}
	0,1,2,3 \mapsto 1,2,4,5
	\end{equation}
	is the following:
	
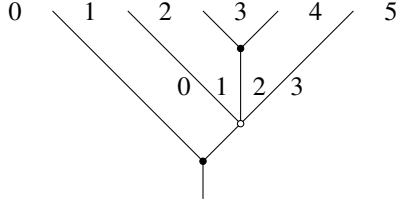
\[
	\begin{tikzpicture}
		\foreach \r in {0,...,5}{
			\draw (\r,0) node{${\r}$};
		}
		\foreach \r in {0,...,3}{
			\draw (\r*.5+2.25,-1) node{${\r}$};
		}
		\draw (2.5,0) -- (3,-.5) -- (3.5,0);
		\draw (3,-.5) -- (3,-1.5) -- (1.5,0);
		\draw (.5,0) -- (2.5,-2) -- (4.5,0);
		\draw (2.5,-2) -- (2.5,-2.5);
		
		\draw[fill] (3,-.5) circle (1.2pt);
		\draw[fill=white] (3,-1.5) circle (1.2pt);
		\draw[fill] (2.5,-2) circle (1.2pt);
	\end{tikzpicture}
	\]
	For a morphism $f: l \to m$ in $\Delta$ and $i \in \{1,\ldots,m\}$, let $\lambda_i$ be the $i$-th leaf of the planar tree corresponding to $f$. Then
	\[
		S^1(f)(k) = 
		\begin{cases}
		j &\text{if $\lambda_i$ is above the $j$-th incoming edge of the white vertex,} \\
		* &\text{if $\lambda_i$ is not above the white vertex.}
		\end{cases}
	\]
	For $n > 1$ and $X \subset \{1,\ldots,m\}$ with $n$ elements,
	\[
		S^n(f)(X) = 
		\begin{cases}
		S^1(f)(X) &\text{if this set is a subset of $\{1,\ldots,l\}$ which has $n$ elements,} \\
		* &\text{otherwise.}
		\end{cases}
	\]
	For example, if $f$ is the map \eqref{equationmorphismindelta}, then $S^2(f)$ is the map
	\[
	12,13,14,15,23,24,25,34,35,45 \mapsto *,*,*,*,12,12,13,*,23,23.
	\]
\end{remark}

\subsection{The Hochschild-Pirashvili cochains}

\begin{definition}
	Let $\mathrm{FSet}_*$ be the (skeletal) category of finite pointed sets. Its objects are non-negative integers and morphisms $l \to m$ are any functions $\{0,\ldots,l\} \to \{0,\ldots,m\}$ sending $0$ to $0$. It is a symmetric monoidal category with monoidal product given by pointed union $\vee$.
\end{definition}

\begin{definition}
	Let $\epsilon: 0 \to 1$, $\mu: 2 \to 1$ and $\rho: 1 \to 0$ be the morphisms in $\mathrm{FSet}_*$ represented in the picture below:
	
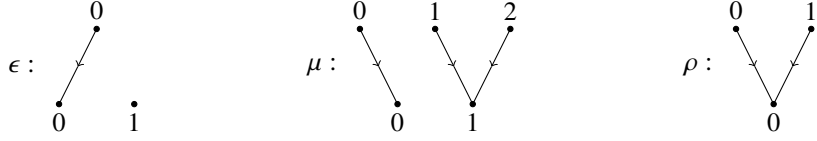
\[
	\begin{tikzpicture}
	\draw (-1,-.5) node{$\epsilon:$};
	\draw[fill] (0,0) circle (1pt) node[above]{$0$};
	\draw[fill] (-.5,-1) circle (1pt) node[below]{$0$};
	\draw[fill] (.5,-1) circle (1pt) node[below]{$1$};
	\draw[->] (0,0) -- (-.25,-.5);
	\draw (-.25,-.5) -- (-.5,-1);
	
	\begin{scope}[shift={(4.5,0)}]
	\draw (-1.5,-.5) node{$\mu:$};
	\draw[fill] (-1,0) circle (1pt) node[above]{$0$};
	\draw[fill] (0,0) circle (1pt) node[above]{$1$};
	\draw[fill] (1,0) circle (1pt) node[above]{$2$};
	\draw[fill] (-.5,-1) circle (1pt) node[below]{$0$};
	\draw[fill] (.5,-1) circle (1pt) node[below]{$1$};
	\draw[->] (-1,0) -- (-.75,-.5);
	\draw (-.75,-.5) -- (-.5,-1);
	\draw[->] (0,0) -- (.25,-.5);
	\draw[->] (1,0) -- (.75,-.5);
	\draw (.25,-.5) -- (.5,-1) -- (.75,-.5);
	\end{scope}
	
	\begin{scope}[shift={(9,0)}]
	\draw (-1,-.5) node{$\rho:$};
	\draw[fill] (0,-1) circle (1pt) node[below]{$0$};
	\draw[fill] (-.5,0) circle (1pt) node[above]{$0$};
	\draw[fill] (.5,0) circle (1pt) node[above]{$1$};
	\draw[->] (-.5,0) -- (-.25,-.5);
	\draw[->] (.5,0) -- (.25,-.5);
	\draw (-.25,-.5) -- (0,-1) -- (.25,-.5);
	\end{scope}
	\end{tikzpicture}
	\]
\end{definition}

\begin{definition}
	For a closed monoidal category $\mathcal{E}$ and $M \in \mathcal{E}$, let $\otimes_M$ be the monoidal structure on the image of the contravariant endofunctor
	\[
	\mathcal{E}(-,M): \mathcal{E} \to \mathcal{E}.
	\]
	given by
	\[
		\mathcal{E}(A,M) \otimes_M \mathcal{E}(B,M) := \mathcal{E}(A \otimes B,M).
	\]
\end{definition}

\begin{definition}\cite[Section 2.9]{BataninBergerLattice}
	Let $\mathcal{E}$ be a symmetric closed monoidal category, $A$ a commutative monoid in $\mathcal{E}$ and $M$ an $A$-module. Let
	\[
		\mathcal{L}(A,M):\mathrm{FSet}_* \to \mathcal{E}
	\]
	be the unique contravariant functor sending $1$ to $\mathcal{E}(A,M)$, $\epsilon$, $\mu$ and $\rho$ to the maps corresponding to the multiplication of $A$, the unit of $A$ and the action of $M$ on $A$ respectively, and which is strong monoidal for the monoidal products $\vee$ and $\otimes_M$.
\end{definition}

\begin{definition}\cite[Section 1.7]{pirashvili}\label{definitionhochschildpirashvili}
	Let $A$ be a commutative algebra, that is a commutative monoid in the category of vector spaces, and $M$ an $A$-module. We consider the covariant functor obtained as the composite of the contravariant functor
	\[
		\Delta \xrightarrow{S^d} \mathrm{FSet}_* \xrightarrow{\mathcal{L}(A,M)} \mathrm{Vect}.
	\]
	The \emph{Hochschild-Pirashvili cochain of order $d$ of $A$ with coefficients in $M$} is defined as the cochain complex associated to this cosimplicial vector space.
\end{definition}

\subsection{Construction of a homotopy left cofinal functor}

\begin{lemma}\label{lemmacofinalfunctor}
	For $n \geq 1$, there is a functor
	\[
	\alpha_n: \Delta \to \Delta\mathcal{K}_n.
	\]
\end{lemma}

\begin{proof}
	Let us recall that, by Definition \ref{definitioncategorydeltap}, $\Delta\mathcal{K}_n$ is the category whose object are operations of $\mathcal{K}_n$. Also recall that a $\mathcal{K}_n$-tree is a combed tree with vertices decorated with operations of $\mathcal{K}_n$. A morphism is given by a black and white $\mathcal{K}_n$-tree with a unique white vertex, together with an extra morphism from the target of this $\mathcal{K}_n$-tree. The source of the morphism is given by the operation which decorates the white vertex. The target is given the target of the extra morphism.
	
	For $m \geq 0$, we define $\alpha_n(m) := (\mu,\sigma) \in \mathcal{K}_n(k)$, where
	\[
	k := \frac{m!}{n!(m-n)!}
	\]
	and $(\mu,\sigma)$ is given as follows. First, note that ${m \choose n}$ is equipped with the lexicographic order. Moreover, there is a uniquely determined order-preserving bijection $\{1,\ldots,k\} \simeq {m \choose n}$. Now for $1 \leq i < j \leq k$, let $\{i_1,\ldots,i_n\}$ and $\{j_1,\ldots,j_n\}$ be the subsets of $\{1,\ldots,m\}$ corresponding to $i$ and $j$ respectively. We define $\mu_{ij} := \lambda$, where $\lambda$ is the lowest element in $\{1,\ldots,n\}$ such that $i_\lambda \neq j_\lambda$.
	
	Now let $f: l \to m$ be a morphism in $\Delta$. Let $g := S^n(f)$. We define $\alpha_n(f)$ as the following black and white $\mathcal{K}_n$-tree. The root vertex is a black vertex decorated with $g^{-1}(*)$. The edge above the root which corresponds to $*$ has a white vertex above it. The white vertex is decorated with $\alpha_n(l)$. The edge above the white vertex corresponding to $S \in \alpha_n(l)$ has a black vertex above it which is decorated with $g^{-1}(S)$, if $g^{-1}(S)$ is not a singleton. The extra linear order on the set of leaves of the tree reorders the fibers. The morphism from the target of this tree to $\alpha_n(m)$ is uniquely determined, since $\mathcal{K}_n(k)$ is a poset. For example, if $f: 3 \to 5$ is the map \eqref{equationmorphismindelta} of Remark \ref{remarksimplicialsphere}, then $\alpha_2(f)$ looks as follows:
	
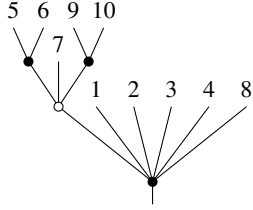
\[
		\begin{tikzpicture}
			\draw (0,-.3) -- (0,0);
			\draw[fill] (0,0) circle (1.7pt);
			\draw (-.25,1) -- (0,0) -- (.25,1);
			\draw (-.75,1) -- (0,0) -- (.75,1);
			\draw (-1.25,1) -- (0,0) -- (1.25,1);
			
			\draw (-1.65,1.6) -- (-1.25,1) -- (-.85,1.6);
			\draw (-1.25,1) -- (-1.25,1.6);
			\draw[fill=white] (-1.25,1) circle (1.7pt);
			
			\draw (-1.85,2.05) -- (-1.65,1.6) -- (-1.45,2.05);
			\draw (-1.05,2.05) -- (-.85,1.6) -- (-.65,2.05);
			\draw[fill] (-1.65,1.6) circle (1.7pt);
			\draw[fill] (-.85,1.6) circle (1.7pt);
			
			\draw (-1.85,2.05) node[above]{\small{$5$}};
			\draw (-1.45,2.05) node[above]{\small{$6$}};
			\draw (-1.25,1.6) node[above]{\small{$7$}};
			\draw (-1.05,2.05) node[above]{\small{$9$}};
			\draw (-.65,2.05) node[above]{\small{$10$}};
			\draw (-.75,1) node[above]{\small{$1$}};
			\draw (-.25,1) node[above]{\small{$2$}};
			\draw (.25,1) node[above]{\small{$3$}};
			\draw (.75,1) node[above]{\small{$4$}};
			\draw (1.25,1) node[above]{\small{$8$}};
		\end{tikzpicture}
	\]
	Indeed, the root vertex has $6$ edges above it corresponding to the $6$ elements in the fibre $g^{-1}(*) = \{*,12,13,14,15,34\}$. The white vertex has $3$ edges above it corresponding to the $3$ elements in $\alpha_2(3)$. The first and the last edges above the white vertex both have a black vertex above it because they correspond to the $12$ and $23$ respectively, and both $g^{-1}(12)$ and $g^{-1}(23)$ have two elements.
	
	It is straightforward to check that $\alpha_n$ is indeed a functor.
\end{proof}

\begin{remark}
	It might help to consider $\alpha_n(m)$ as the configuration of $\frac{m!}{n!(m-n)!}$ points in $\mathbb{R}^n$. For example, when $m=6$ and $n=2$, the configuration of points looks as follows:
	
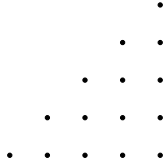
\[
	\begin{tikzpicture}[scale=.5]
	\foreach \r in {1,...,5}{
		\foreach \s in {1,...,\r}{
			\draw[fill] (\r,\s) circle (1.6pt);
		}
	}
	\end{tikzpicture}
	\]
	Moreover, one can construct a function $\pi$ from the space of configuration of $k$ points in $\mathbb{R}^n$ to $\mathcal{K}_n(k)$. This function sends $(x_1,\ldots,x_k)$ to $(\mu,\sigma)$, where
	\begin{itemize}
		\item $\mu_{ij}$ is the smallest $\lambda \in \{1,\ldots,n\}$ such that the $\lambda$-th coordinates of $x_i$ and $x_j$ are different,
		\item $\sigma_{ij}=id$ if and only if $x_i < x_j$.
	\end{itemize}
	Then $\alpha_n(m)$ is the image of the configuration of $\frac{m!}{n!(m-n)!}$ points described above through $\pi$.
\end{remark}

\begin{lemma}
	The functor $\alpha_n$ of Lemma \ref{lemmacofinalfunctor} is homotopy left cofinal.
\end{lemma}

\begin{proof}
	Assume we have an object of $\Delta\mathcal{K}_n$, given by $k \geq 0$ and $p \in \mathcal{K}_n(k)$. To simplify the notation, let $\mathcal{C} := \alpha_n / p$. We want to prove that $\mathcal{C}$ is contractible. The objects of $\mathcal{C}$ are given by $n \geq 0$ together with a morphism $\alpha_n(m) \to p$. Such morphism is given by a black and white $\mathcal{K}_n$-tree with a unique white vertex decorated with $\alpha_n(m)$, together with an extra morphism from the target of this tree to $p$.
	
	Let $\mathcal{C}_0$ be the full subcategory of $\mathcal{C}$ of black and white $\mathcal{K}_n$-trees whose root vertex is the unique white vertex. $\mathcal{C}_0$ is contractible because it has an initial object. Note that $\alpha_n(n) \in \mathcal{K}_n(1)$ is the unit of the operad $\mathcal{K}_n$. The initial object of $\mathcal{C}_0$ is given by $m=n$ and the morphism $\alpha_n(n) \to p$ is the following $\mathcal{K}_n$-tree. The root vertex is white and decorated with the unit of $\mathcal{K_n}$. It is a unary vertex and the unique vertex above it is black and decorated with $p$. The extra morphism is the identity on $p$.
	
	Now let us prove that $\mathcal{C}_0$ is a reflective subcategory of $\mathcal{C}$. Let us define the left adjoint $L$ to the inclusion $\mathcal{C}_0 \to \mathcal{C}$. Assume that we have an object $c \in \mathcal{C}$ given by $m \geq 0$ and a morphism $x: \alpha_n(m) \to p$. Let $r$ be the root vertex of the $\mathcal{K}_n$-tree given by the morphism $x$. If $r$ is white, then we can define $L(c) := c$. Otherwise, $r$ is black. In this case, let $l$ be the number of edges above $r$. $r$ is decorated with an object $(\mu,\sigma) \in \mathcal{K}_n(l)$. Let $i \in \{1,\ldots,l\}$ be such that the $i$-th edge above $r$ is below the unique white vertex of the $\mathcal{K}_n$-tree $x$. We have a partition of the set $\{1,\ldots,l\}$ into $2n+1$ subsets. The first subset is given by $\{i\}$ the $2n$ other subsets are given as follows. For $\mu' \in \{1,\ldots,n\}$ and $\sigma' \in \Sigma_2$, the subset corresponding to the pair $(\mu',\sigma')$ contains all the elements $j \in \{1,\ldots,l\}$ such that $j \neq i$, $\mu_{ij}=\mu'$ and $\sigma_{ij}=\sigma'$. Also, for a pair $(\mu',\sigma')$ as above, there is an operation of $o_{\mu',\sigma'}$ of $\mathcal{K}_n$ given by restricting the complex graph $(\mu,\sigma)$ to the subset corresponding to $(\mu',\sigma')$. Now we define the black and white $\mathcal{K}_n$-tree $L(c)$ as follows. The root vertex is the unique white vertex and is decorated with $\alpha_n(m+2)$. Each edge above it corresponds to a subset of $\{0,\ldots,m+1\}$ with $n$ elements. Let $S=\{i_1,\ldots,i_n\}$ be such subset. First we consider the case where $S$ does not contain $0$ or $n+1$. Then the black vertex above the edge corresponding to $S$ is defined as the black vertex above the edge corresponding to $S$ in $c$. Now we consider the case where $S$ is of the form $\{0,\ldots,n\} \backslash \{\mu'\}$ for some $\mu' \in \{1,\ldots,n\}$. Then the black vertex above the edge corresponding to $S$ is decorated with $o_{\mu',id}$. Next, we consider the case where $S$ is of the form $\{m-n+1,\ldots,m+1\} \backslash \{m-n+\mu'\}$ for some $\mu' \in \{1,\ldots,n\}$. Then the black vertex above the edge corresponding to $S$ is decorated with $o_{\mu',(12)}$. Finally, we consider the case where $S$ contains $0$ or $n+1$ but is not contained in $\{0,\ldots,n\}$ or $\{m-n+1,\ldots,m+1\}$. Then the black vertex above the edge corresponding to $S$ is nullary. This defines $L(c)$. The unit of the adjunction is given by the morphism $m \to m+2$ in $\Delta$ sending $i$ to $i+1$. The counit is the identity. In conclusion, $\mathcal{C}_0$ is a reflective subcategory of $\mathcal{C}$ and $\mathcal{C}_0$ is contractible. This proves that $\mathcal{C}$ is also contractible.
\end{proof}

\subsection{Equivalence between the two definitions of Hochschild cochains}

\begin{theorem}\label{theoremequivalencehochschild}
	Let $A$ be a commutative algebra. Then the Hochschild object $\mathrm{Hoch}(A)$ of Definition \ref{definitionhochschildobject}, when $A$ is seen as a $\mathcal{K}_n$-algebra, coincides with the Hochschild-Pirashvili cochain of order $n$ of $A$.
\end{theorem}

\begin{proof}
	The following square commutes:
	\begin{equation}\label{equationcommutativesquare}
		\xymatrix{
			\Delta \ar[r]^-{S^n} \ar[d]_-{\alpha_n} & \mathrm{FSet}_*^{op} \ar[d]^-{\mathcal{L}(A,A)} \\
			\Delta \mathcal{K}_n \ar[r]_-{\mathrm{End}_A} & \mathrm{Cat}
		}
	\end{equation}
	For any small category $\mathcal{C}$ and any functor $X: \mathcal{C} \to \mathrm{Cat}$, let $\delta: \mathcal{C} \to \mathrm{Cat}$ be the functor given by $\delta(c) := \mathcal{C}/c$ and
	\[
		\mathrm{Tot}(X) := \underline{\mathrm{Hom}}_\mathcal{C} (\delta,X).
	\]
	We have
	\begin{equation}\label{equationhochschild}
		\mathrm{Hoch}(A) = \mathrm{Tot} ( \mathrm{End}_A ) = \mathrm{Tot} (\mathrm{End}_A \cdot \alpha_n) = \mathrm{Tot}\left(\mathcal{L}(A,A) \cdot S^n\right),
	\end{equation}
	where the first equality is by Definition \ref{definitionhochschildobject} of $\mathrm{Hoch}(A)$, the second one uses Lemma \ref{lemmacofinalfunctor} and \cite[Theorem 19.6.7 (2)]{hirschhorn}, and the last one is thanks to the commutativity of the square \eqref{equationcommutativesquare}. The conclusion follows from the fact that conormalization can be viewed as totalization \cite[Section 3.11]{BataninBergerLattice}. This means that the last term in \eqref{equationhochschild} coincides with Definition \ref{definitionhochschildpirashvili}.
\end{proof}

\begin{corollary}
	For a commutative algebra $A$, there is a $\mathcal{K}_{n+1}$-action on the Hochschild-Pirashvili cochain of order $n$ of $A$.
\end{corollary}

\begin{proof}
	According to Corollary \ref{corollaswisscheeseaction}, there is an action of $\mathcal{RK}_{n+1}$ on the pair $(A,\mathrm{Hoch}(A))$. In particular, there is a $\mathcal{K}_{n+1}$-action on $A$. The conclusion follows from Theorem \ref{theoremequivalencehochschild}.
\end{proof}

\bibliographystyle{plain}
\bibliography{swiss-cheese}

\end{document}